\documentclass{article}
\usepackage{amsmath, amssymb, amscd,amsthm,amsfonts,amscd}
\usepackage{graphics, graphicx}
\usepackage{float,epsfig,color}
\newtheorem{theorem}{Theorem}[section]
\newtheorem{lemma}[theorem]{Lemma}
\newtheorem{proposition}[theorem]{Proposition}
\newtheorem{corollary}[theorem]{Corollary}
\newtheorem{definition}[theorem]{Definition}

\newtheorem{remark}[theorem]{Remark}

\newtheorem*{acknowledgements}{Acknowledgements}

\newcommand{\CC}{\mathcal{C}}
\newcommand{\HH}{\mathbb{H}}

\newcommand{\C}{\mathbb{C}}
\newcommand{\R}{\mathbb{R}}

\newcommand{\multeq}{multiple equilibrium point}
\newcommand{\eqpt}{equilibrium point}

\newcommand{\Res}{\textrm{Res}}

\newcommand{\vf}{\frac{d}{dz}}

\newcommand{\codim}{{\rm codim}}
\begin{document}
\title{On Parameter Space of Complex Polynomial Vector Fields in $\C$}
\author{Kealey Dias \vspace{0.2cm}\\ {\small (with an appendix by Dias and Tan Lei)}}
\date{\today}
\maketitle
\begin{abstract}
The space $\Xi_d$ of degree $d$ single-variable monic and centered complex polynomial vector fields can be  decomposed into loci in which the vector fields have the same topological structure. We analyze the geometric structure of these loci and describe some bifurcations, in particular, it is proved that new homoclinic separatrices can form under small perturbation. By an example, we show that this decomposition  of parameter space by combinatorial data is not a cell decomposition.\newline
The appendix to this article, joint work with Tan Lei, shows that landing separatrices are stable under small perturbation of the vector field if the multiplicities of the equilibrium points are preserved.
\end{abstract}
%% %%%%%%%%%%%%% subject classification
\begin{figure}[b]%
\flushleft
\rule[0mm]{54.0mm}{0.15mm}\\ \hspace*{5.0mm}%
{\footnotesize  2010 Mathematics Subject Classification: 37F75, 34Cxx, 34C23, 34C37.\par
 }
\end{figure}
%-----------------------
\section{Introduction}
The objects we consider are the vector fields in $\C$  that in a global chart take the form $P(z)\vf$, with $P(z)=z^d+a_{d-2}z^{d-2}+\cdots +a_0$, with $z$ and $a_i \in \C$. We are interested in the global qualitative dynamics of the integral curves of these vector fields, or equivalently, solutions to the real-time, first order ordinary differential equations $\dot{z}=P(z)$ (with $P$ as above), where the dot is the derivative with respect to time, $t\in \R$. The space $\Xi_d \simeq \mathbb{C}^{d-1}$ of these vector fields of
degree $d$ can be decomposed into loci $\mathcal{C}$ in which the
vector fields have the same \emph{combinatorial data set} (to be defined). We will prove that each
of these combinatorial classes is a connected manifold with well-defined (real)
dimension $q$, which is the dimension of the combinatorial class as a subspace in $\Xi_d$. \par
We present now a summary of some necessary concepts and definitions. It can be shown that $\infty$ is a pole of order $d-2$ for vector fields $\xi_P \in \Xi_d$. There are $2(d-1)$ trajectories $\gamma_{\ell}$ which meet at infinity with asymptotic angles $  \frac{2 \pi\ell}{2\left(d-1\right)}$, $\ell \in \{ 0,1,\dots,2d-3 \}$.
When the labelling index $\ell$ is even,  the trajectories are called \emph{incoming} to $\infty$, and when the index $\ell$ odd, they are called \emph{outgoing} from $\infty$ (see Figure \ref{trajsatinfty2}).\par
There are $2d-2$ accesses to $\infty$ defined by the trajectories at infinity. An \emph{end} $e_{\ell}$ is infinity with access between $\gamma_{\ell-1}$ and $\gamma_{\ell}$ (see Figure \ref{trajsatinfty2}). An \emph{odd end} is an end $e_k$ labelled by an odd index $k$, and an \emph{even end} is an end $e_j$ labelled by an even index $j$.\par
 \begin{figure}%
    \centering
    \resizebox{!}{5cm}{\input{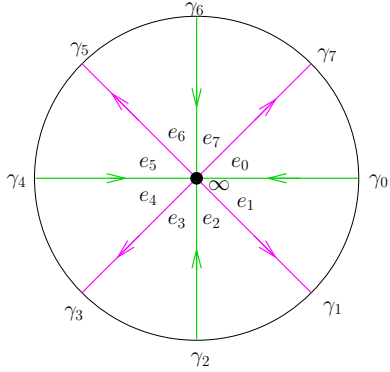}}
    \caption{The point at $\infty$ is a pole of order $d-2$ for vector fields $\xi_P \in \Xi_d$. There are $2(d-1)$ trajectories $\gamma_{\ell}$ which meet at infinity with asymptotic angles $  \frac{2 \pi\ell}{2\left(d-1\right)}$, $\ell \in \{ 0,1,\dots,2d-3 \}$. There are $2d-2$ accesses to $\infty$ defined by the trajectories at infinity. An \emph{end} $e_{\ell}$ is infinity with access between $\gamma_{\ell-1}$ and $\gamma_{\ell}$. An \emph{odd end} is an end $e_k$ labelled by an odd index $k$, and an \emph{even end} is an end $e_j$ labelled by an even index $j$.}
    \label{trajsatinfty2}
 \end{figure}
\emph{Separatrices}  $s_{\ell}$  are the maximal trajectories of $\xi_P$ incoming to and outgoing from $\infty$ (in finite time). They are labelled also by the $2(d-1)$ asymptotic angles.
A separatrix $s_{\ell}$ is called \emph{landing} if $\bar{s}_{\ell}\setminus s_{\ell}=\zeta$, where $\zeta$ is an equilibrium point for $\xi_P$ (equivalently, a zero of $P$). A separatrix
$s_{\ell}=s_{k,j}$ is called \emph{homoclinic} if $\bar{s}_{k,j}\setminus s_{k,j}=\emptyset$. See Figures \ref{centerzoneshade}, \ref{sepalzoneshade}, and \ref{alphaomegashade} for some examples of landing and homoclinic separatrices. A separatrix for a polynomial vector field $\xi_P \in \Xi_d$ can only be either homoclinic or landing. A homoclinic separatrix $s_{k,j}$ is labelled by the one odd index $k$ and the one even index $j$ corresponding to its two asymptotic directions at infinity.
The \emph{Separatrix graph:} $\Gamma_P=\bigcup \limits_{\ell=0}^{2d-3}\hat{s}_{\ell}$, that is, the union of all separatrices and any equilibrium points at which they land, as well as the point at infinity. completely determines the topological structure of the trajectories of a vector field (see, for instance, \cite{DN1975}, \cite{ALGM1973}).\par
%----------------------------------------------
\subsection{Zones}
The connected components $Z$ of $\mathbb{C}\setminus \Gamma_P$ are called \emph{zones}.
There are three types of zones for vector fields in $\Xi_d$, and the types of zones are determined by the types of their boundaries:
 \begin{itemize}
 \item
 A \emph{center zone} $Z$ contains an equilibrium point, which is a center, in its interior. Its boundary consists of one or several homoclinic separatrices and the point at infinity. If a center zone is on the left of $n$ homoclinic separatrices $s_{k_1,j_1},\dots,s_{k_n,j_n}$ on the boundary $\partial Z$, then the center zone has $n$ odd ends $e_{k_1},\dots,e_{k_n}$ at infinity on $\partial Z$ and the zone is called either a \emph{counter-clockwise center zone} or an \emph{odd center zone}. If a center zone is on the right of $n$ homoclinic separatrices $s_{k_1,j_1},\dots,s_{k_n,j_n}$ on the boundary $\partial Z$, then the center zone has $n$ even ends $e_{j_1},\dots,e_{j_n}$ at infinity on $\partial Z$ and the zone is called either a \emph{clockwise center zone} or an \emph{even center zone} (see Figure \ref{centerzoneshade}).
  \item
 A \emph{sepal zone} $Z$ has exactly one equilibrium point on the boundary, which is both the $\alpha$-limit point and $\omega$-limit point for all trajectories in $Z$ (i.e. $\zeta_{\alpha}= \zeta_{\omega}$). This equilibrium point is necessarily a multiple equilibrium point. The boundary $\partial Z$ contains exactly one incoming and one outgoing landing separatrix,  the point at infinity, and possibly one or several homoclinic separatrices.   If a sepal zone is to the left of $n$ homoclinic separatrices
$s_{k_1,j_1},\dots,s_{k_n,j_n}$ on its boundary, then it has $n+1$ odd ends on the
boundary: $e_{k_1},\dots,e_{k_n}$ and $e_{j_i+1}$ for some corresponding $j_i$,
depending on how one orders the separatrices. In this case, it is called an \emph{odd} sepal zone. Similarly, if a sepal zone is on the right of
$n$ homoclinic separatrices $s_{k_1,j_1},\dots,s_{k_n,j_n}$ on
its boundary, then it has $n+1$ even ends on the boundary, $e_{j_1},\dots,e_{j_n}$
and $e_{k_i+1}$ for some corresponding $k_i$, again depending on the ordering of
the separatrices. In this case, it is called an \emph{even} sepal zone (see Figure \ref{sepalzoneshade}).
\item
An \emph{$\alpha \omega$-zone} $Z$ has two equilibrium points on the boundary, $\zeta_{\alpha}\neq \zeta_{\omega}$, the $\alpha$-limit point and $\omega$-limit point for all trajectories in $Z$. The boundary $\partial Z$ contains one or two incoming landing separatrices and one or two outgoing landing separatrices, possibly one or several homoclinic separatrices, and the point at infinity. If an ${\rm \alpha \omega}$-zone is both on the left of $n_1$ homoclinic separatrices
$s_{k_1,j_1},\dots,s_{k_{n_1},j_{n_1}}$ and on the right of $n_2$ homoclinic
separatrices $s_{k_1,j_1},\dots,s_{k_{n_2},j_{n_2}}$ on the boundary, then the ${\rm
\alpha \omega}$-zone has $n_1+1$ odd ends ($e_{k_1},\dots,e_{k_{n_1}}$ and
$e_{j_i+1}$ for some corresponding $j_i$) and $n_2+1$ even ends
($e_{j_1},\dots,e_{j_{n_2}}$ and $e_{k_i+1}$ for some corresponding $k_i$) on the
boundary (see Figure \ref{alphaomegashade}).
 \end{itemize}
 \begin{remark}
 It will be important to note for an $\alpha \omega$-zone, there are exactly one odd end and one even end, neither of whose indices coincide with any index of a homoclinic separatrix (in the notation above the odd and even ends are $e_{j_i+1}$ and $e_{k_i+1}$ respectively).
 \end{remark}
   %%%%----------------------
\subsection{Transversals}
There are several ways to encode the combinatorial structure of a vector field. The author's preferred descriptions rely on objects called \emph{transversals}.  We define in this section the important structures needed to understand  definitions of a \emph{combinatorial data set}.\par
In any simply connected domain avoiding zeros of $P$, the differential $\frac{dz}{P(z)}$ has an antiderivative,
unique up to addition by a constant
\begin{equation}
\Phi(z)=\int_{z_0}^{z} \frac{dw}{P(w)}.\nonumber
\end{equation}
Note that
\begin{equation}
\Phi_{\ast}\left(\xi_P\right)=\Phi'\left(z\right)P\left(z\right)\frac{ d}{ dz}=\frac{ d}{ dz}.
\end{equation}
The coordinates $w=\Phi(z)$ are, for this reason, called \emph{rectifying coordinates}.
%----------------------------------------
We will call the images of zones under rectifying coordinates \emph{rectified zones}. The rectified zones and corresponding boundaries are of the following types:
\begin{itemize}
\item
 The image of a center zone (minus a curve contained in the zone which joins the center $\zeta$ and $\infty$) under $\Phi$ is a vertical half strip. It is an upper vertical half strip for a counterclockwise center zone, and the odd ends and homoclinic separatrices are mapped to the lower boundary. It is a lower vertical half strip for a clockwise center zone, and the even ends and homoclinic separatrices are mapped to the upper boundary of this half strip (see Figure \ref{centerzoneshade}).\par
%%%%
  \begin{figure}%
    \centering
    \resizebox{!}{6cm}{\input{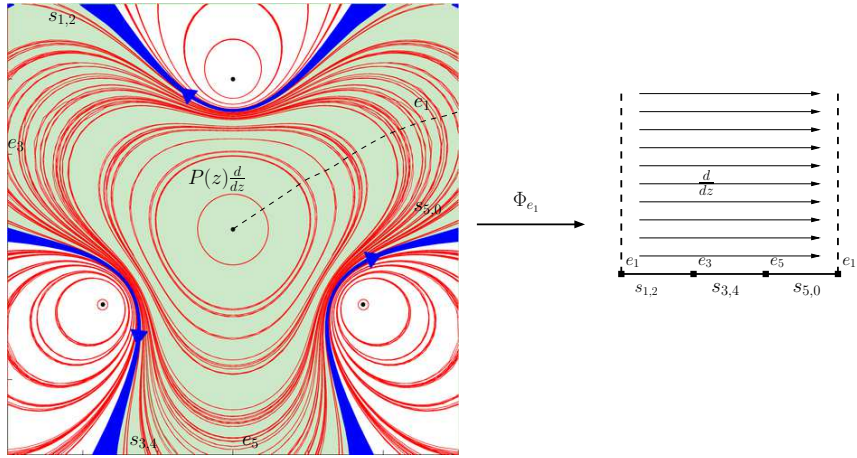}}
    \caption{Pictured are the trajectories of a vector field with four center zones: one odd center zone (shaded) with homoclinic separatrices $s_{5,0}$, $s_{1,2}$, and $s_{3,4}$ and ends $e_1$, $e_3$, and $e_5$ on the boundary; and three even center zones, each with one homoclinic separatrix and one end on the boundary. The image of a center zone (minus a curve contained in the zone which joins the center $\zeta$ and $\infty$) under $\Phi$ is a vertical half strip. It is an upper vertical half strip for a counterclockwise center zone, and  a lower vertical half strip for a clockwise center zone. In this figure, there is an odd center zone mapped to an upper vertical half strip.}
    \label{centerzoneshade}
 \end{figure}
%%%%
\item
The image of an odd sepal zone under $\Phi$ is an upper half plane, where odd ends and homoclinic separatrices are mapped to the lower boundary of this half plane. The image of an even sepal zone under $\Phi$ is a lower half plane, where even ends and homoclinic separatrices are mapped to the upper boundary of this half plane (see Figure \ref{sepalzoneshade}).
%%%%
  \begin{figure}%
    \centering
    \resizebox{!}{6cm}{\input{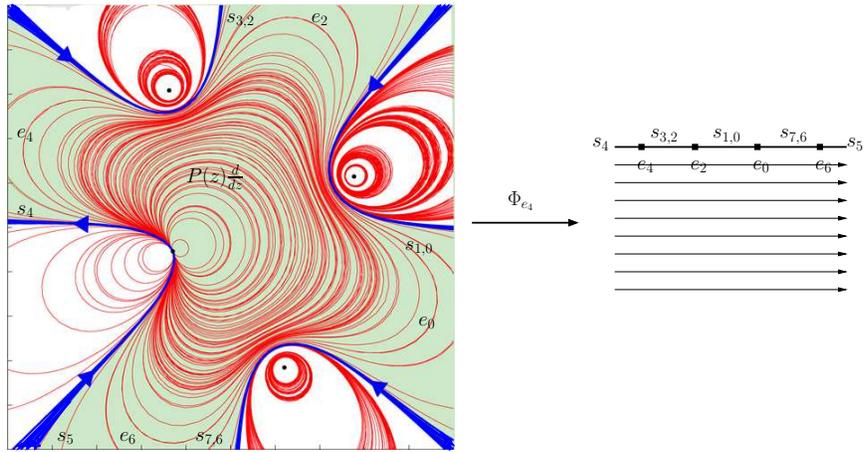}}
    \caption{Pictured are the trajectories of a vector field with an even sepal zone (shaded). On the boundary of the sepal zone is the double equilibrium point which is both the $\alpha$ and $\omega$ limit point of the trajectories; one incoming landing separatrix $s_4$ and one outgoing landing separatrix $s_5$; three homoclinic separatrices $s_{1,0}$, $s_{3,2}$, and $s_{7,6}$; and four ends at infinity $e_0$, $e_2$, $e_4$, and $e_6$. There is an odd sepal zone (not shaded) which shares the equilibrium point and the landing separatrices with the shaded sepal zone, but it has no homoclinic separatrices and only one odd end $e_5$ on the boundary. The image of an odd sepal zone under $\Phi$ is an upper half plane, and the image of an even sepal zone is a lower half plane. In this figure, there is an even sepal zone mapped to a lower half plane.}
    \label{sepalzoneshade}
 \end{figure}
%%%%
\item
The image of an $\alpha \omega$-zone under $\Phi$ is a horizontal strip (see Figure \ref{alphaomegashade}). The lower boundary of the strip consists of two landing separatrices, odd ends, and counterclockwise homoclinic separatrices on the boundary of the zone. The upper boundary of the strip consists of two landing separatrices, even ends, and clockwise homoclinic separatrices on the boundary of the zone.
%%%%
  \begin{figure}%
    \centering
    \resizebox{!}{6cm}{\input{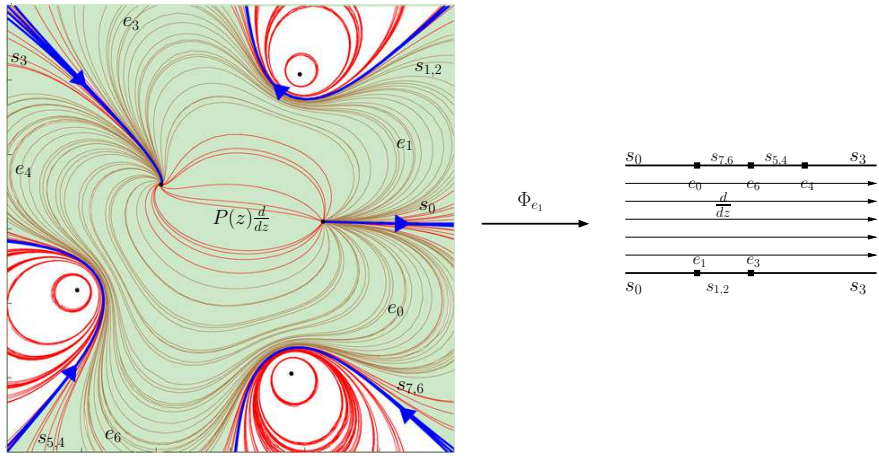}}
    \caption{Pictured are the trajectories of a vector field with an $\alpha \omega$-zone (shaded). On the boundary of the  zone are two equilibrium points: one which is the $\alpha$-limit point of the trajectories and the other is the $\omega$-limit point of the trajectories. Also on the boundary are one incoming landing separatrix $s_0$ and one outgoing landing separatrix $s_3$. The zone  is to the left of the homoclinic separatrix $s_{1,2}$  and on the right of the two  homoclinic separatrices $s_{5,4}$, and $s_{7,6}$. Finally, the boundary contains two odd ends $e_1$ and $e_3$ and three even ends $e_0$, $e_4$, and $e_6$ at infinity. The image of an $\alpha \omega$-zone under $\Phi$ is a horizontal strip.}
    \label{alphaomegashade}
 \end{figure}
%%%%
\end{itemize}
 Via the rectifying coordinates, it is evident that there are a number of closed geodesics through $ \infty$ in $\hat{\mathbb{C}} \setminus \{\text{equilibrium pts}  \}$ in the metric with length element $\frac{|\rm{d}z|}{|P(z)|}$. Among these are the $h$ homoclinic separatrices, and there are $s$ \emph{distinguished transversals} (defined below).
 \begin{definition}
  The \emph{distinguished transversal} $T_{k,j}$ is the geodesic in the metric $\frac{|\rm{d}z|}{|P(z)|}$ joining the ends $e_k$ and $e_j$, avoiding the separatrices and equilibrium points, where $e_j$ is the left-most end on the upper boundary and $e_k$ is the right-most end on the lower boundary of the strip that is the image of the $\alpha \omega$-zone in which the transversal is contained (see Figure \ref{transstrip}).
 \end{definition}
 Note that the way in which the distinguished transversal is chosen, the indices of the ends it joins are exactly those ends whose indices will never coincide with the indices of any homoclinic separatrices.
%%%%
  \begin{figure}
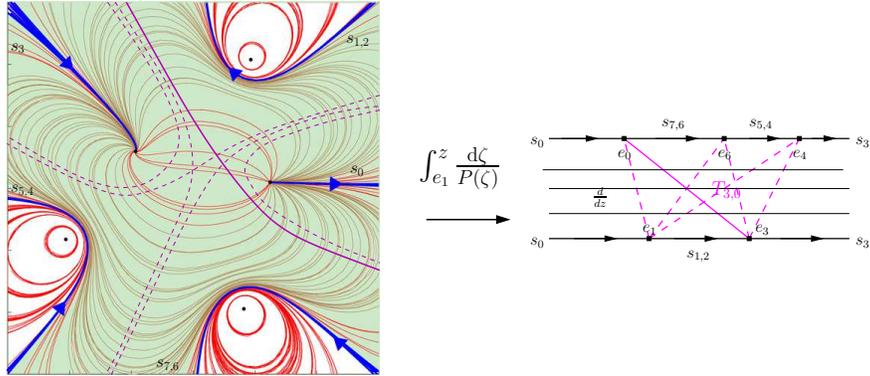
%
    \centering
    \begin{minipage}{5cm}
    \resizebox{!}{5cm}{\input{transversals2_defensepres.pstex_t}}
    \end{minipage}
    \hspace{.25cm}
    \begin{minipage}{1cm}
    \resizebox{!}{1cm}{\input{rectmap.pstex_t}}
    \end{minipage}
    \hspace{.25cm}
    \begin{minipage}{2cm}
    \resizebox{!}{2cm}{\input{transversals_diststrip.pstex_t}}
    \end{minipage}
    \caption{ Each $\alpha \omega$-zone is isomorphic to a strip. There may be several transversals which avoid the equilibrium points and separatrices (the dashed curves), but there is exactly one distinguished transversal for each $\alpha \omega$-zone (in this case, $T_{3,0}$). We define the distinguished transversal to be the geodesic in the metric $|dz|/|P(z)|$ joining the ends $e_k$ and $e_j$ (in this figure, $e_3$ and $e_0$) where $e_j$ is the left-most end on the upper boundary of the strip and $e_k$ is the right-most end on the lower boundary of the strip. Since these indices are the same as the indices for the two landing separatrices on the upper left and lower right boundary of the strip, they can never coincide with the indices for a homoclinic separatrix.}
    \label{transstrip}
 \end{figure}
%%%%
%-------------------------
\subsection{Combinatorial and Analytic Data}
One way to describe the topological structure of a vector field is by the union of homoclinic separatrices $s_{k,j}$ and distinguished transversals $T_{k,j}$. It was proved in \cite{KDcomb} that this description is equivalent to the one presented in the classification (from \cite{BD09}). Essentially, we want to use the numbers $\mathbb{Z}/(2d-2)$ to stand for indices of separatrices for homoclinics, and indices of ends for distinguished transversals otherwise.  The indices of transversals were chosen in a way to never conflict with the indices of homoclinic separatrices.
A \emph{combinatorial data set} can be described as a bracketing on the string $0 \ 1 \ 2\dots 2d-3$, where the elements paired by parentheses
correspond to the labels of the separatrices or distinguished transversals we want to pair. Round parentheses $(\cdots)$ are used to mark pairings corresponding to homoclinic separatrices, square parentheses $[\cdots]$ are used to mark pairings corresponding to a
distinguished transversal in each $\alpha \omega$-zone, and elements that are not paired correspond to the ends in sepal-zones (see Figure \ref{d4wsepalnhom} for some examples).\par
%---------------------
\begin{figure}
\resizebox{!}{4.2cm}{\input{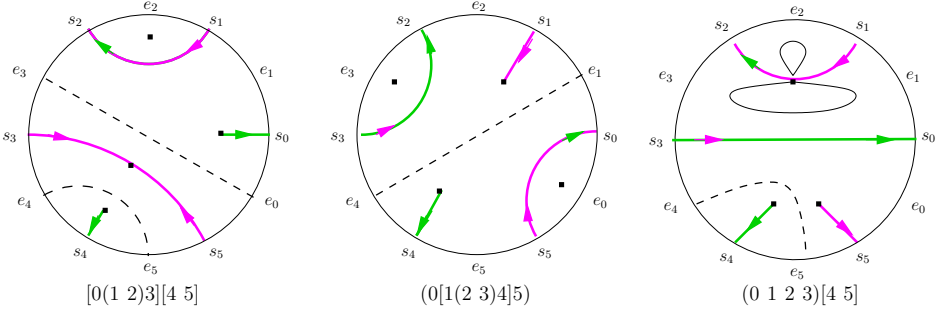}}
\caption{Disk models for three examples of vector fields of degree $d=4$ having
sepal zones or/and homoclinic separatrices.  The pairing of the ends is marked by
the dashed curves. The representation of the combinatorics in brackets is
displayed below each figure.}
\label{d4wsepalnhom}
\end{figure}
%---------------------
The \emph{analytic invariants} are an $(s+h)$-tuple in $\HH^s\times \mathbb{R}_+^h$ where to each homoclinic separatrix is assigned
a number $\tau=\int_{s_{k,j}}\frac{dz}{P(z)}>0$, and to each distinguished transversal is assigned a number $\alpha=\int_{T}\frac{dz}{P(z)}\in \HH$. \par
Putting the combinatorial and analytic data together, one can uniquely describe a vector field in $\Xi_d$ by a metric graph with a single vertex (corresponding to the pole at infinity), $h$ solid loops (corresponding to homoclinic separatrices), and $s$ dashed loops (corresponding to distinguished transversals). Each of the solid loops is assigned a positive real number and each dashed loop is assigned a complex number, corresponding to the analytic invariants (see Figure \ref{metricgraph}).
%---------------------
\begin{figure}
\centering
\Large
\resizebox{!}{3.5cm}{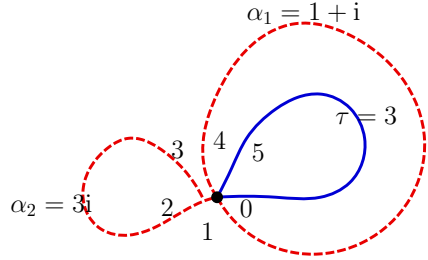}
\caption{Example of possible metric graph defining a complex polynomial vector field. The combinatorics can be described by the bracketing $(0 [1[2 \ 3]4]5)$ and the analytic invariants are the $(2+1)$-tuple $(  1+{\rm i},  3{\rm i},3 )\in \mathbb{H}_+^2 \times \R_+^1$.}
\label{metricgraph}
\end{figure}
%---------------------
Such a metric graph is a complete set of realizeable invariants for the classification of these vector fields (\cite{BD09}). Another interesting fact about this presentation is that each connected component of the plane minus this transversal flower contains exactly one equilibrium point.\par
We decompose parameter space into classes $\mathcal{C}$ of vector fields that have the same separatrix graph with the labeling. \par
Our goal is to understand bifurcations of the global trajectory structure, which means we need to understand changes in separatrix structure under small perturbation. In this paper, we will partially answer this question. That is, we will pick and arbitrary $\xi_{P_0}\in \Xi_d$, and try to answer which classes $\CC$ intersect every arbitrarily small neighborhood of $\xi_{P_0}$.
%----------------------
\begin{acknowledgements}
This research was supported by a grant from Idella Fonden and by the Marie Curie European Union Research Training Network {\it Conformal Structures and Dynamics} (CODY). Support for this project was also provided by a PSC-CUNY Award, jointly funded by The Professional Staff Congress and The City University of New York.
\end{acknowledgements}
%------------------------
\section{Topological and Analytic Structure of the Loci}
\label{analstrata}
The classification in \cite{BD09} gives a bijection between a combinatorial class $\mathcal{C}$ and $\HH^s \times \R_+^h$. The following theorem proves the type of bijection.
\begin{theorem}
\label{analstratastr}
There exists a real analytic isomorphism $G_{\mathcal{C}}:\HH^s \times \R_+^h\rightarrow \mathcal{C}$, which is $\C$-analytic in the first $s$ coordinates and $\R$-analytic in the last $h$ coordinates. It is the restriction of a holomorphic mapping in $(s+h)$ complex variables: $\tilde{G}_{\mathcal{C}}: \HH^s \times V_{\R_+}^h(\epsilon)\rightarrow \tilde{\mathcal{C}}$, where $\tilde{\mathcal{C}} \supset \mathcal{C}$.
\end{theorem}
In particular, each $\mathcal{C}$ is naturally foliated by $\C$-analytic leaves of complex dimension $s$.
\begin{proof}
We prove, that $G_{\mathcal{C}}$  is a restriction of a holomorphic function in $s+h$ variables.
%%\framebox[1.1\width]{Is $\mathbb{R}$-analytic iff restriction of  $\mathbb{C}$-analytic?}
Let
\begin{equation}
V_{\R_+}(\epsilon)=\{z \mid \Re(z)>0, |\Im (z)|<\epsilon  \},
\end{equation}
for $\epsilon$ sufficiently small.
We prove the existence of a holomorphic function
\begin{align}
G: \HH^s \times V_{\R_+}^h(\epsilon) &\rightarrow \Xi_d \nonumber \\
\underline{\alpha} &\mapsto \xi_{\underline{\alpha} },
\end{align}
$\underline{\alpha}=(\alpha_1,\dots,\alpha_s,\tau_1,\dots,\tau_h)\in \HH^s \times V_{\R_+}^h(\epsilon)$. \par
By Hartog's Theorem, it is enough to show that $G$ is holomorphic in each $\alpha$ and $\tau$ in the single-variable sense (we drop the indices on the $\alpha$ and $\tau$ to simplify notation) in order to conclude $G$ is holomorphic in the $(s+h)$-variable sense. We will construct  families of surfaces $\bar{\mathcal{M}}_{\alpha}$, $\bar{\mathcal{M}}_{\tau}$ and maps $G_{\alpha}$, $G_{\tau}$ such that
\begin{align}
G_{\alpha}&:\bar{\mathcal{M}}_{\alpha}\rightarrow \bar{\C}, \quad \left( G_{\alpha} \right)_{\ast}\left(\vf\right)=\xi_{\alpha}, \quad \xi_{\alpha} \in \Xi_d,\\
 G_{\tau}&: \bar{\mathcal{M}}_{\tau} \rightarrow \bar{\C}, \quad \left( G_{\tau} \right)_{\ast}\left(\vf\right)=\xi_{\tau},\quad \xi_{\tau} \in \Xi_d,
\end{align}
and we will prove that each family is holomorphic in the one complex variable $\alpha$ or $\tau$ by utilizing holomorphic dependence of parameters in the Measurable Riemann Mapping Theorem (\cite{AB1960}). \par
We first define the \emph{rectified surface} $\mathcal{M}_0(\mathcal{C})$ associated to the vector field $\xi_0 \in \mathcal{C}$ and with analytic invariant the $(s+h)$-tuple $\underline{\alpha}_0=(\alpha^0_1,\dots,\alpha^0_s,\tau^0_1,\dots,\tau^0_h)$. Without loss of generality, we can take $\underline{\alpha}_0=(\rm{i},\dots,\rm{i},1,\dots,1)$ to simplify presentation. This combinatorial class has a number of rectified zones $Z$ with analytic invariants $\underline{\alpha}_0$ (see the left side of Figure \ref{affinemaps}).
  %%%%
   \begin{figure}%
     \centering
     \resizebox{!}{8cm}{\input{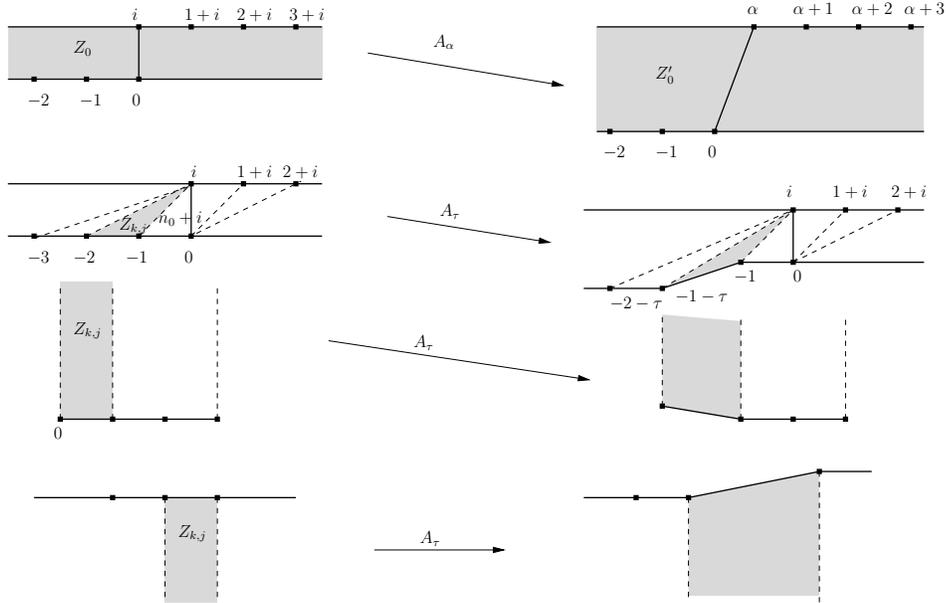}}
     \caption{Some examples of canonical rectified zones (left) and their images under $A_{\alpha}$ or $A_{\tau}$, the distorted rectified zones (right). }
    \label{affinemaps}
  \end{figure}
  %%%%%%
  Each separatrix has exactly two representations on the boundary of the rectified zones: one on the upper boundary of a rectified zone and one representation on the lower boundary of a (possibly the same) rectified zone. There are also several representations of $\infty$ on the boundaries of the rectified zones, called the \emph{ends}. Let
\begin{equation}
\mathcal{M}^{\ast}_0(\mathcal{C}):= \left( \bigsqcup Z \right)/\sim,
\end{equation}
where $\sim$ is the appropriate identification of the two representations of each separatrix and the identification of all ends. We let $\mathcal{M}_0(\mathcal{C}) \simeq \bar{\C}$ be the compactification (for details, see \cite{BD09}).\par
We now define \emph{distorted rectified surfaces} $\mathcal{M}_{\alpha}(\mathcal{C})$ and $\mathcal{M}_{\tau}(\mathcal{C})$ respectively by the following.
We consider first the case where we allow one $\alpha_0=\rm{i}$ to vary. Choose the strip $Z_{0}$ associated to $\alpha_0$. Choose a complex number $\alpha \in \HH$. We define a piecewise affine mapping $A_{\alpha}$, $\alpha \in \HH$, on the rectified zones $Z$ as follows. Let $A_{\alpha}$ be the piecewise affine mapping which is the identity on all  rectified zones $Z\neq Z_{0}$, and on $Z_{0}$, it is defined by $\rm{i} \mapsto \alpha$ and $1 \mapsto 1$. Then on $Z_{0}$
\begin{equation}
A_{\alpha}(z)=\frac{1}{2}(1-{\rm i} \alpha)z +\frac{1}{2}(1+{\rm i} \alpha)\bar{z} ,
\end{equation}
%\begin{equation}
%A_{\alpha}(z)=\frac{1}{2 \Im (\alpha_0)}({\rm i} \bar{\alpha}_0-{\rm i}\alpha)z +\frac{1}{2 \Im (\alpha_0)}(-{\rm i}\alpha_0+{\rm i}\alpha)\bar{z} .
%\end{equation}
 The mapping $A_{\alpha}$ maps $Z_{0}$ onto the \emph{distorted rectified zone} $Z_{0}':=A_{\alpha}(Z_{0}) $. As before, we define $\mathcal{M}_{\alpha}(\mathcal{C})$ as the compactification of
 \begin{equation}
\mathcal{M}^{\ast}_{\alpha}(\mathcal{C}):=\left( Z'_{0}\sqcup\bigsqcup \limits_{Z\neq Z_{0}}Z\right)/\sim.
 \end{equation}
The argument is similar but slightly more complicated for $\mathcal{M}_{\tau}(\mathcal{C})$, where we allow exactly one $\tau_0 \in \mathbb{R}_+$ to vary. A homoclinic separatrix $s_{k,j}$ is on the boundary of exactly two zones, so we have two rectified zones $Z_1$ and $Z_2$ with the rectified homoclinic separatrix $s_{k,j}$ with length $\tau_0 =1$ on their boundaries that will be distorted when we allow $\tau_0$ to vary holomorphically (see the right-hand side of Figure \ref{affinemaps}). These two zones can be a combination of strips, half-planes, and vertical half-strips (cylinders). For an $s_{k,j}$ on the boundary of either an upper (respectively lower) half-plane or vertical half-strip, let $Z_{k,j}$ be the  vertical half-strip of width $\tau_0=1$, such that $s_{k,j}$ is on the boundary. We let $A_{\tau}$, $\tau \in V_{\R_+}(\epsilon)$,  be the piecewise affine map that is the identity on all  rectified zones $Z\neq Z_1$ or $Z_2$ and the identity (perhaps with some translation) on $Z_1 \setminus Z_{k,j} $ or $Z_2 \setminus Z_{k,j}$, and on $Z_{k,j}$, it is defined by $1 \mapsto \tau\in V_{\R_+}$ and $\pm {\rm i}  \mapsto \pm {\rm i}$. This affine map takes the form
\begin{equation}
A_{\tau}(z)=\frac{1}{2}(\tau+1)z + \frac{1}{2}(\tau-1)\bar{z}.
\end{equation}
If $s_{k,j}$ is on the lower boundary of a strip, then we distort the strip by a (three-piece) piecewise mapping by the construction below. Details are included for completeness, but the idea is much easier to understand by consulting Figures \ref{stripdeform1} and \ref{stripdeform2}.
%%%%
  \begin{figure}%
    \centering
    \resizebox{!}{5cm}{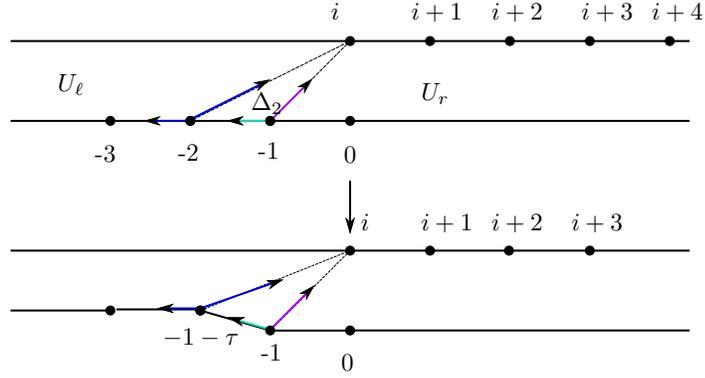}
    \caption{The triangle $\Delta_2$ has vertices $\rm{i}$, $-2$, and $-1$ on the boundary of the strip. Let $U_{\ell}$ be the part of the strip to the left of $\Delta_2$ and $U_r$ to the right. If we distort some $\tau_0=1$ on the lower edge of $\Delta_2$, then on $U_{\ell}$, the affine map $A_{\tau}$ is defined by $-1 \mapsto -1$ and $2+{\rm{i}}\mapsto \tau + 1 +{\rm{i}}$. On $\Delta_2$, the affine map is defined by $-1\mapsto -\tau$ and $1+\rm{i}\mapsto 1+\rm{i}$. On $U_r$, $A_{\tau}$ is the identity.}
    \label{stripdeform1}
 \end{figure}
%%%%
%%%%
  \begin{figure}%
    \centering
    \resizebox{!}{5cm}{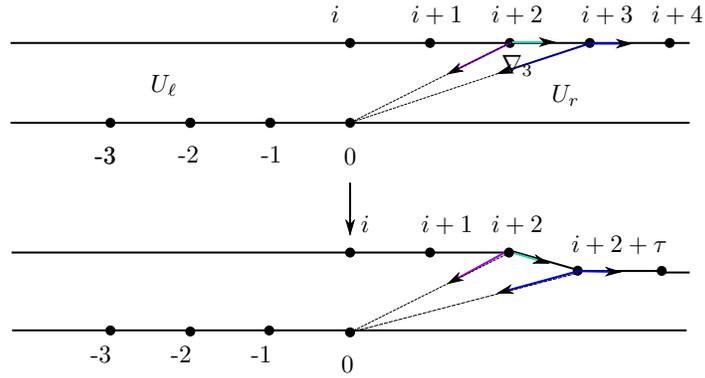}
    \caption{The triangle $\nabla_3$ has vertices $0$, ${\rm{i}}+2$, and ${\rm{i}}+3$ on the boundary of the strip. Let $U_{\ell}$ be the part of the strip to the left of $\nabla_3$ and $U_r$ to the right. If we distort some $\tau_0=1$ on the upper edge of $\nabla_3$, then on $U_{\ell}$, the affine map $A_{\tau}$ is the identity.  On $\nabla_3$, the affine map is defined by $1\mapsto \tau$ and $-2-\rm{i}\mapsto -2-\rm{i}$. On $U_r$, $A_{\tau}$ is defined by $1 \mapsto 1$ and $-3-{\rm{i}}\mapsto -\tau -2-{\rm{i}}$.}
    \label{stripdeform2}
 \end{figure}
%%%%
Let $\Delta_j$ be the triangle in the strip with vertices $\rm{i}$, $-j$, and $-j+1$ on the boundary of the strip. One edge of $\Delta_j$ is on the lower boundary of the strip. Let $\nabla_j$ be the triangle in the strip with vertices $0$, ${\rm{i}}+(j-1)$, and ${\rm{i}}+j$. One edge of $\nabla_j$ is on the upper boundary of the strip. In either case, let $U_{\ell}$ be the part of the strip to the left of either $\Delta_j$ or $\nabla_j$, and $U_r$ to the right. If we distort some $\tau_0=1$ on the lower edge of some $\Delta_j$, then on $U_{\ell}$, the affine map $A_{\tau}$ is defined by $-1 \mapsto -1$ and $j+{\rm{i}}\mapsto \tau + j-1 +{\rm{i}}$. This corresponds to the affine map
\begin{equation}
A_{\tau}(z)=\frac{1}{2}(2+{\rm{i}}-{\rm{i}}\tau)z+\frac{1}{2}(-{\rm{i}}+ {\rm{ i}} \tau)\bar{z}.
\end{equation}
On $\Delta_j$, the affine map is defined by $-1\mapsto -\tau$ and $j-1+\rm{i}\mapsto j-1+\rm{i}$. This corresponds to the affine map
\begin{equation}
A_{\tau}(z)=\frac{1}{2}(\tau +1 +{\rm{i}}(j-1)[\tau-1])z+\frac{1}{2}(\tau -1 -{\rm{i}}(j-1)[\tau-1])\bar{z}.
\end{equation}
On $U_r$, $A_{\tau}$ is the identity. The construction is similar for $\nabla_j$.\par
The mapping $A_{\tau}$ sends  $Z_1$ and $Z_2$ to the \emph{distorted rectified zones} $Z'_1:=A_{\tau}(Z_1)$ and $Z'_2:=A_{\tau}(Z_2)$. As before, we define $\mathcal{M}_{\tau}(\mathcal{C})$ as the compactification of
 \begin{equation}
\mathcal{M}^{\ast}_{\tau}(\mathcal{C}):= \left((Z'_1 \sqcup Z'_2)\sqcup\bigsqcup \limits_{Z \neq Z_1, Z_2}Z\right) /\sim.
 \end{equation}
 The exact expressions defining $A_{\alpha}$ and $A_{\tau}$ are not so important.  What is important is that the associated Beltrami coefficients $\mu_{\alpha}$ and $\mu_{\tau}$ are holomorphic in $\alpha$ and $\tau$ respectively and satisfy $\| \mu_{\alpha}\|_{\infty}<1$, $\| \mu_{\tau}\|_{\infty}<1$.  \par
 We endow $\mathcal{M}_{\alpha}$ with the standard complex structure $\sigma_0$ and the vector field $\vf$. We pullback by $A_{\alpha}$, giving us a new complex structure $\sigma_{\alpha}$ in $\mathcal{M}_0$ that depends holomorphically on $\alpha$. The rectifying coordinates extend %by Morera's theorem
  to a mapping $\phi:\C \rightarrow \left(\bigsqcup \limits_Z\bar{Z}\right)/\sim$. Under pullback, we induce a new almost complex structure $\tilde{\sigma}_{\alpha}$ in $\C$, analytic in $\alpha$, and a vector field $(\phi \circ A_{\alpha})^{\ast}\left(\vf\right)$ in $\C$. \par
 Let $\zeta^0_i$ be the \eqpt s of $(\phi \circ A_{\alpha})^{\ast}\left(\vf\right)$. By the Measurable Riemann Mapping Theorem (MRMT), there exists a family of quasiconformal maps $f_{\alpha}: \C \rightarrow \C$, normalized such that
\begin{align}
&f_{\alpha}(\infty)=\infty,\\
&\sum\limits_{i}f_{\alpha}(\zeta^0_i)=0,\quad \text{and}\\
&f_{\alpha}(s_0) \ \text{is asymptotic to} \ \R_+,
\end{align}
 such that $(f_{\alpha})^{\ast}\sigma_0=\tilde{\sigma}_{\alpha}$. The mapping $G_{\alpha}=(f_{\alpha}\circ \phi^{-1} \circ A_{\alpha}^{-1})$ is holomorphic in $z$, and by MRMT, $f_{\alpha}$ is holomorphic in $\alpha$.  Then $(G_{\alpha})_{\ast}\left(\vf\right)=P_{\alpha}(z) \vf$, where $P_{\alpha}$ is holomorphic in $\C$. The above is summarized in the diagram
\[
\begin{CD}
\left(\mathcal{M}_{\alpha} ,\sigma_0,\vf\right) @<A_{\alpha} <<\left(\mathcal{M}_0,\sigma_{\alpha},A_{\alpha}^{\ast}(\vf)\right)\\
@VV G_{\alpha} V @AA \phi A\\
(\C, \sigma_0,P_{\alpha}\vf) @<f_{\alpha} << (\C,\tilde{\sigma}_{\alpha},(\phi \circ A_{\alpha})^{\ast}(\vf)).
\end{CD}
\]
The index of the vector field at infinity is $-(d-2)$ (look about the \emph{ends} in rectifying coordinates), so infinity must be the only pole of order $d-2$ for the vector field. We can conclude that $P$ is a degree $d$ polynomial, which by the above normalizations is monic and centered.
So for fixed $\alpha$, $P_{\alpha}$ takes the form $P_{\alpha}(z)= \prod \limits_{i=1}^d(z-\zeta_i)$, $\zeta_i=f_{\alpha}(\zeta^0_i)$.  We need to show that $P_{\alpha}$ is holomorphic in $\alpha$, and it is enough to show that the $\zeta_i$  are analytic functions of $\alpha$. We can conclude that the roots $\zeta_i$ are analytic functions of $\alpha$ since $f_{\alpha}$ is holomorphic in $\alpha$ for fixed $z$. \par
Therefore, $G$ is holomorphic in each $\alpha,\ \tau$ and is hence holomorphic in $(s+h)$ complex variables. Therefore, $G$ is an open mapping. The restriction $G_{\mathcal{C}}:\HH^s \times \R_+^h\rightarrow \mathcal{C}$ is an open mapping and is furthermore bijective by the classification in \cite{BD09}. Hence $G_{\mathcal{C}}$ is an isomorphism which is $\C$-analytic in the first $s$ coordinates, and $\R$-analytic in the last $h$ coordinates.
\end{proof}
\begin{corollary}[Corollary of Theorem \ref{analstratastr}]
\label{connectedstrata}
Each $\mathcal{C}$ is connected. The (real) dimension of each $\mathcal{C}$ is $\dim_{\R}(\mathcal{C})=2s+h$, and the codimension (with respect to $\Xi_d$) is $\codim_{\R}(\mathcal{C})=2(d-1)-(2s+h)=2m^{\ast}+h$.
\end{corollary}
\begin{remark}
By Corollary \ref{connectedstrata} and the enumeration of combinatorial classes in \cite{KDcomb}, we know exactly how many loci there are altogether and how many loci there are of a particular dimension.
\end{remark}
\subsection{Cone Structure of Loci}
\label{conestrsec}
Each combinatorial class $\mathcal{C}\cong \HH^s \times \R_+^h$ is an $\R_+$ cone with $z^d \vf \in \Xi_d$ as base point.
We need the following proposition stated in Pilgrim \cite{KPdessin}.
\begin{proposition}
\label{cone}
Let $P(z)=\prod \limits_{j=1}^d(z-\zeta_j)$ and $\mathcal{C} \ni \xi_P$. For every $c >0$, $\xi_{\tilde{P}}\in \mathcal{C}$ for $\tilde{P}(z)=\prod \limits_{j=1}^d(z-c\zeta_j) $.
\end{proposition}
\begin{proof}
If $\gamma(t)$ is a real trajectory of the vector field given by $P(z)$, i.e. $\gamma'(t)=P(\gamma(t))$, then for every $c>0$, $\eta(t)=c \gamma(c^{d-1}t)$ is a real trajectory of the vector field given by $\tilde{P}(z)$. Indeed, $\eta'(t)=c^d\gamma'(c^{d-1}t)=c^dP(c^{d-1}t)=\tilde{P}(c \gamma(c^{d-1}t))=\tilde{P}(\eta(t))$. Since $c>0$, the trajectories $\eta(t)$ are reparameterizations by time of the $\gamma(t)$, preserving orientation.
\end{proof}
\begin{corollary}
\label{minstratumadj}
The minimal stratum $\underline{0} \in \C^{d-1}$ corresponding to the vector field $z^d \vf$ is adjacent to all other loci.
\end{corollary}
\begin{proof}
We use Lemma \ref{cone}, note that $\tilde{P}$ is continuous in $c$, and let $c \rightarrow 0$.
\end{proof}
This cone structure is also reflected in the analytic invariants for a class. We note what happens to the analytic invariants when roots of the polynomial are multiplied by the constant $c$.
\begin{proposition}
\label{analinvsc}
The analytic invariants for $\tilde{P}$ are equal to $1/c^{d-1}$ times the analytic invariants for $P$.
\end{proposition}
\begin{proof}
The $\Res(1/P,\zeta)$ are conformal invariants (Brickman and Thomas \cite{BT76}), so the analytic invariants are too. Therefore, $\tilde{P}\vf$ has the same analytic invariants as $c^{d-1}P \vf$, and
\begin{equation}
\tilde{\alpha}=\int_{\gamma}\frac{dz}{c^{d-1}P(z)}=\frac{1}{c^{d-1}}\alpha,
\end{equation}
where $\tilde{\alpha}$ and $\alpha$ are the corresponding analytic invariants for $\tilde{P}\vf$ and $P\vf$ respectively.
\end{proof}
%---------------------
\section{Structural Stability and Bifurcations}
It is natural to consider the possible bifurcations for these vector fields. More specifically, we want to understand: given an arbitrary $\xi_0\in \Xi_d$, which combinatorial classes intersect every arbitrarily small neighborhood of $\xi_0$. So we need to consider changes in the separatrix structure for small perturbations of $\xi_0$.
\subsection{Structurally Stable Vector Fields}
\begin{proposition}
The structurally stable vector fields in $\Xi_d$ are the vector fields with neither multiple equilibrium points nor homoclinic separatrices.
\end{proposition}
\begin{proof}
The vector fields without homoclinic separatrices or \multeq s are structurally stable, which follows immediately from Theorem \ref{landingstablethm}. By Theorem \ref{analstratastr}, the vector fields with either a homoclinic separatrix or \multeq \  form loci of dimension strictly less than the maximal dimension and must therefore belong to the bifurcation locus.
\end{proof}
\begin{corollary}
\label{strstbdense}
The structurally stable vector fields are dense in $\Xi_d$.
\end{corollary}
%%--------------------
In general, bifurcations can be complicated when we allow multiple \eqpt s to split (see Section \ref{nocelldecomp} for an example). We therefore consider first the possible bifurcations when the multiplicities of the equilibrium points are preserved under small perturbation.
Theorem \ref{landingstablethm} in the appendix of this paper proves that a landing separatrix cannot be lost under small perturbation when preserving multiplicity, so the non-splitting bifurcations must be those involving only breakings of one or more homoclinic separatrices.
%----------------------------------
\subsection{Some Non-splitting Bifurcations}
\label{bifsfxmult}
The construction in the proof of Theorem \ref{analstratastr} in fact tells us more than is stated in the theorem. Since the only non-splitting bifurcations can involve breakings of homoclinic separatrices, all non-splitting bifurcations can be understood by analyzing the combinatorics of the deformed zones.
%% Is this clear since non-splitting bifs can only involve breaking of homos and V^h covers all possible homo splittings?
We describe in the following certain non-splitting bifurcations, and an exhaustive analysis of these is to be considered in a future paper.\par
We start by explaining what can happen if exactly one analytic invariant associated to a homoclinic separatrix is allowed to take values in $\pm \HH$, instead of being restricted to $\R_+$, while the rest of the analytic invariants are preserved. That is, we consider the possible bifurcations when exactly one homoclinic separatrix $s_{k,j}$ breaks. A homoclinic separatrix $s_{k,j}$ is on the boundary of exactly two zones. We consider the \emph{distorted zones} as the proof of Theorem \ref{analstratastr}, where we allow $\tau_0(s_{k,j})\mapsto \tau \in V_{\R_+}(\epsilon)\setminus \R_+$. We know that the distorted zones endowed with the vector field  $\vf$ correspond to some monic and centered polynomial vector fields in  a neighborhood of the given combinatorial class. When we allow  a single $\tau_0(s_{k,j})$ to vary holomorphically, then this causes the separatrices $s_k$ and $s_j$ to land.  If $\tau \in +\HH$, then instead of coming back into infinity (resp. outgoing from) infinity, the separatrix $s_k$ (resp. $s_j$) now lands at the \eqpt \ on the boundary of the zone having $s_{k,j}$ as part of its upper (resp. lower) boundary in rectifying coordinates.
If $\tau \in -\HH$,  then instead of coming back into (resp. outgoing from) infinity, the separatrix $s_k$ (resp. $s_j$) now lands at the \eqpt \ on the boundary of the zone having $s_{k,j}$ as part of its lower (resp. upper) boundary in rectifying coordinates.
The \eqpt \ at which $s_k$ (resp. $s_j$) lands is either a sink (resp. source) or \multeq , depending on whether the lower or upper (resp. upper or lower) rectified zone having $s_{k,j}$ on the boundary was a vertical half-strip or strip in the first case, or in the latter case, a half plane.  Notice that if $s_{k,j}$ is on the boundary of a vertical half-strip, then the associated center becomes either a sink or source. See Figure \ref{affinemaps2} for some examples.\par
% %%%%
   \begin{figure}%
     \centering
     \resizebox{!}{8cm}{\input{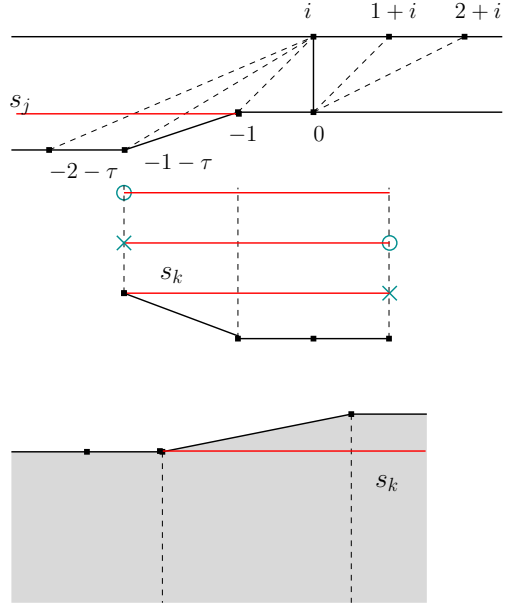}}
     %\resizebox{!}{8cm}{\input{affinemaps2.pstex_t}}
     \caption{Some examples of distorted zones endowed with the vector field $\vf$. The separatrices (drawn in red) are the trajectories going out of and coming into the ends, so for a distorted zone, these may not necessarily be on the boundary of the distorted zones. Each of the two separatrices which were a homoclinic separatrix for the non-distorted zones  enter opposite zones on which the homoclinic separatrix was part of the boundary. In the top picture, $s_j$ lands at either the source or \multeq \ to which the strip is associated. In the middle picture, $s_{k,j}$ belonged to the lower boundary of an upper half-strip, and after perturbation, $s_k$ now lands at the \eqpt \ which was on the boundary of the half-strip, making the center a sink (one should see the points marked by the blue crosses and circles in the figure as being identified). In the bottom picture, $s_k$ now lands at the \multeq \ which had $s_{k,j}$ on the uppper boundary of one of its associated half-planes.}
    \label{affinemaps2}
  \end{figure}
% %%%%
 If we  allow more than one analytic invariant associated to a homoclinic separatrix to vary at the same time, more complicated things can happen. In particular, new homoclinic separatrices can form. In order to understand this situation, we need to define $H$-chains.
These $H$-chains turn out to be the structures we need to understand exactly which homoclinic separatrices can form under small perturbation, so we define them here
\begin{definition}
An \emph{$H$-chain} of length $n$ is a sequence of $n$ consecutive homoclinic separatrices $\{ s_{k_i,j_i}\}$, $i=1,\dots,n$, i.e. homoclinic separatrices $s_{k_i,j_i}$ such that for each $i$, either $k_{i+1}=j_i+1$ (upper) or $k_{i+1}=j_i-1$ (lower). In particular, a sequence $s_{k_i,j_i}$ such that $k_{i+1}=j_i+1$ for all $i$ is called a \emph{clockwise $H$-chain}, and a sequence $s_{k_i,j_i}$ such that $k_{i+1}=j_i-1$ for all $i$ is called a \emph{counter-clockwise $H$-chain}.
\end{definition}
\begin{remark}
Note that any counter-clockwise $H$-chain is necessarily contained in the lower boundary of a single zone, and a clockwise $H$-chain is contained in the upper boundary of a single zone.
\end{remark}
\begin{definition}
A \emph{closed $H$-chain} of length $n$ is an $H$-chain in which $s_{k_{i+n},j_{i+n}}=s_{k_{i},j_{i}}$, for all $i=1,\dots,n$. An \emph{open $H$-chain} is one that is not closed.
\end{definition}
\begin{remark}
The separatrices in an open $H$-chain have a natural ordering, according to the ordering from left to right in the rectifying coordinates (direction of the flow). The separatrices in a closed $H$-chain do not have a well-defined ordering.
\end{remark}
We first explain the situation where $s_{k,j_0}$ and $s_{k_0,j}$ have a clockwise $H$-chain in common.  We number the $H$-chain with these separatrices at the edges: $s_{k,j_0}=s_{k_1,j_1}, \ s_{k_2,j_2}, \dots , s_{k_n,j_n}=s_{k_0,j}$. The separatrix $s_{k,j}$ forms under small perturbation if and only if all partial sums satisfy
$T_{m}:=\sum \limits_{i=1}^m \Im(\tau_i) >0$, for all $m=1,\dots n-1$ and $T_n=0$ (see Figure \ref{newhom}).\par
 % %%%%
   \begin{figure}%
     \centering
     \resizebox{!}{3cm}{\input{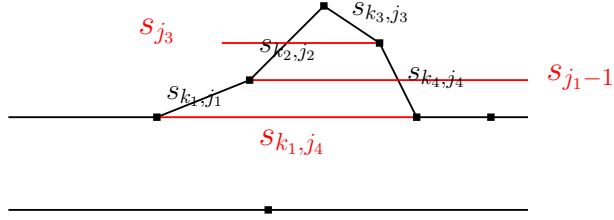}}
     \caption{An example of a distorted strip whose corresponding perturbed vector field has a homoclinic separatrix $s_{k_1,j_4}$ which was not there before perturbation. The separatrices $s_{j_3}$ and $s_{j_1-1}$ land at the \eqpt s in the other zones which had the homoclinic separatrices $s_{k_4,j_4}$ and $s_{k_2,j_2}$ on the upper or lower boundary before perturbation. From the figure, it seems we are distorting the $\tau(s_{k_i,j_i})$ by a non-trivial amount, but they should all be seen as having imaginary part distorted by some small $\epsilon_i$.}
    \label{newhom}
  \end{figure}
% %%%%
It is also possible for a homoclinic separatrix to form under small perturbation if the two initial homoclinics are on the boundary of different zones.
 \begin{proposition}
The separatrix $s_{k,j}$ can form under small perturbation if and only if $s_{k,j_0}$ and $s_{k_0,j}$ have an $H$-chain in common (belong to some $H$-chain), and for an open $H$-chain, $s_k$ is to the left of $s_j$.
\end{proposition}
\begin{proof}
Either $s_{k,j_0}$ and $s_{k_0,j}$ belong to a closed $H$-chain, in which case we can define an $H$-chain such that $s_{k,j_0}$ is to the left of $s_{k_0,j}$; if they do not belong to some closed $H$-chain, then we assume for an open $H$-chain that $s_{k,j_0}$ is to the left of $s_{k_0,j}$. This gives a natural ordering of an $H$-chain with $s_{k,j_0}$ and $s_{k_0,j}$ at it's ends: $s_{k,j_0}=s_{k_1,j_1}, \ s_{k_2,j_2}, \dots , s_{k_n,j_n}=s_{k_0,j}$. For $i=2,\dots,n$, there is a sequence $I_i$ of length $n-1$ with elements in $\{ +, -\}$ corresponding to whether $k_{i+1}=j_i\pm 1$, $i=1,\dots,n-1$. We consider $I_1$ not defined. If there are $q$ sign changes in this itinerary, then the $H$-chain can be decomposed into a sequence of $q+1$ clockwise and counterclockwise $H$-chains, which overlap on the ends (see Figure \ref{Hchain1}). We can then allow $s_{k,j}$ to form by the following conditions on perturbations of the associated $\tau_i$, $i=1,\dots,n$.
For $i=1,\dots,n-1$, if $I_{i+1}=+$, then $\sum \limits_{j=1}^{i} \Im (\tau_i) <0$;  if $I_{i+1}=-$, then $\sum \limits_{j=1}^{i} \Im (\tau_i) >0$; and $\sum \limits_{j=1}^{n} \Im (\tau_i) =0$ (see Figure \ref{Hchain2}). If $s_{k,j_0}$ and $s_{k_0,j}$ do not have an $H$-chain in common, then there is no overlapping sequence of zones through which $s_k$ can have access to $s_j$.
\end{proof}
% %%%%
   \begin{figure}%
     \centering
     \resizebox{!}{4cm}{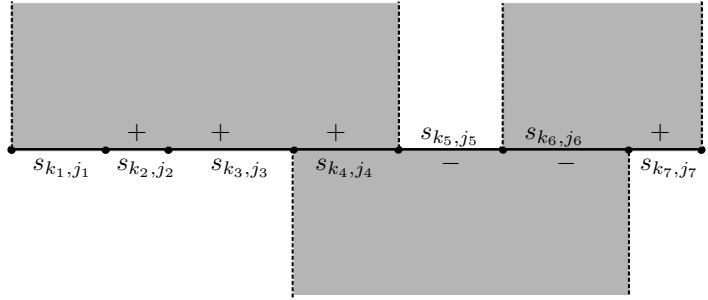}
     \caption{There a natural ordering of an $H$-chain with $s_{k,j_0}$ and $s_{k_0,j}$ at it's ends: $s_{k,j_0}=s_{k_1,j_1}, \ s_{k_2,j_2}, \dots , s_{k_7,j_7}=s_{k_0,j}$. In this example, the sequence $I_i$ for $i=2,\dots,n$ is $I=+,+,+,-,-,+$. We consider $I_1$ not defined. There are $2$ sign changes in this itinerary, so there are three zones corresponding to the three  counterclockwise and clockwise  $H$-chains, which overlap on the ends. }
    \label{Hchain1}
  \end{figure}
% %%%%
% %%%%
   \begin{figure}%
     \centering
     \resizebox{!}{4cm}{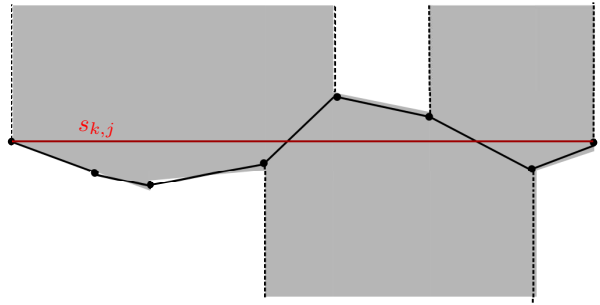}
     \caption{For the $H$-chain as in Figure \ref{Hchain1},  $s_{k,j}$ can form if the appropriate conditions on partial sums of perturbations of the associated $\tau_i$, $i=1,\dots,7$ are satisfied. }
    \label{Hchain2}
  \end{figure}
% %%%%
In general, several homoclinic separatrices can form simultaneously under small perturbation. An exhaustive analysis of the non-splitting bifurcations is an aim of future work.\par
%%-----------------------------
 %----------------------------
\section{No Cell-decomposition}
\label{nocelldecomp}
It turns out that stratifying parameter space by combinatorial invariants does not lead to a cell-decomposition of parameter space.
In general, $\mathcal{C}_0\cap \partial \mathcal{C} \neq \emptyset \nRightarrow \mathcal{C}_0\subset \partial \mathcal{C}$.  We show this by showing that two loci of the same dimension can be adjacent, as demonstrated by the following example. \par
Consider the slice of the combinatorial class $\CC_0 \in \Xi_4$ having combinatorial invariant $[0 \ 1 ]2[3 \ 4]5$ (see Figure \ref{Nocelldecomp2}).
%\begin{equation}
%H=\emptyset, \quad [1]=\{1,2,4,5\}, \quad [0]=\{0\}, \quad [3]=\{3\},
%\end{equation}
Note that  $\dim (\CC_0)=2s_0+h_0=2(2)+0=4$.\par
   \begin{figure}[h]%
 \centering
 \huge
     \resizebox{!}{4cm}{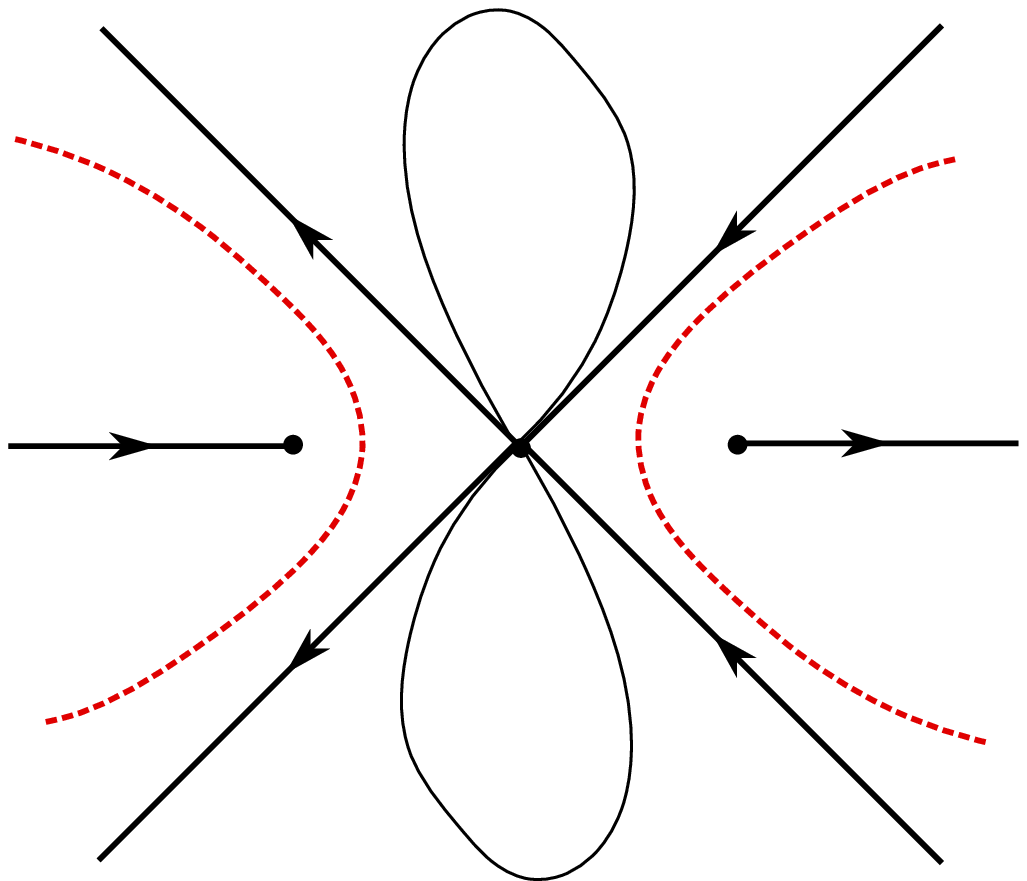}
 \caption{A class $\mathcal{C}_0$ with combinatorics $[0 \ 1 ]2[3 \ 4]5$ having a double equilibrium point and two simple equilibrium points $\zeta_{[0]}$ and $\zeta_{[3]}$. There are two $\alpha \omega$-zones and no homoclinic separatrices, so $\dim (\CC_0)=2s_0+h_0=2(2)+0=4$.}
 \label{Nocelldecomp2}
 \end{figure}
This combinatorial class is adjacent to the combinatorial class $\CC$ having combinatorial invariant $[0(1 \ 2)3](4 \ 5)$,
%\begin{equation}
%H=\{ \{1,2\},\{4,5 \}\},  \quad [0]=\{0\}, \quad [3]=\{3\},
%\end{equation}
and $\dim (\CC)=2s+h=2(1)+2=4$  (see Figure \ref{Nocelldecomp}).  If in $\CC \cong \HH \times \R_+^2$, set $\tau_1=\tau_2 =x$ and let  $\Re(\alpha)=-x$, let $x \rightarrow \infty$.  Then we go to the boundary of the class $\mathcal{C}$ while the residues $\Res (1/P,\zeta_{[0]})=\tau_1+\alpha $ and $\Res (1/P,\zeta_{[3]})=-\tau_2-\alpha $ stay fixed   and the residues  $\Res(1/P,\zeta_{ [2]})$  and $\Res(1/P,\zeta_{ [5]})$ for the centers having $s_{1,2}$ and $s_{5,4}$ respectively on the boundaries of their basins go to infinity. By Lemma \ref{mstarinc} below, at least two points must collide, but these include  neither $\zeta_{[0]}$ nor $\zeta_{[3]}$. This shows that   $\mathcal{C}_0\cap \partial \mathcal{C} \neq \emptyset $. Since these two loci have the same dimension, $\mathcal{C}_0\not\subset \partial \mathcal{C}$.
\begin{lemma}
\label{mstarinc}
If we stay in a bounded subset of any combinatorial class $\mathcal{C}$, i.e. the roots of $P$ stay bounded, then $\Res (1/P,\zeta)\rightarrow \infty$ if and only if $|\zeta - \zeta_i|\rightarrow 0$ for at least one other root $\zeta_i$.
%and one or more $|\alpha|, \ |\tau|\rightarrow \infty$, then
%\[
%m^{\ast}=\sum\limits_{[\ell]\subseteq L } p_{[\ell]}/2
%\]
% increases, $p_{[\ell]}$ defined in Section \ref{eqcl}.
\end{lemma}
\begin{proof}
%Every analytic invariant is the sum of some residues, so if some $|\alpha|$ or $|\tau|\rightarrow \infty$, at least one $\Res (1/P,\zeta) \rightarrow \infty$.
Each residue $\Res(1/P,\zeta)$
is a rational function of the $(\zeta-\zeta_i)$, whose denominator has strictly larger degree than the numerator and takes the form
\begin{equation}
\left( \prod \limits_{i=1}^{d-m}(\zeta-\zeta_i)\right)^{2(m-1)},
\end{equation}
where some of the $\zeta_i$ might be identical and $m$ is the multiplicity of $\zeta$.
 Since by assumption the $|\zeta-\zeta_i|<\infty$, then $\Res (1/P,\zeta) \rightarrow \infty$ if and only if the denominator $\rightarrow 0$, i.e.  at least one of the $(\zeta-\zeta_i)\rightarrow 0$.
\end{proof}
The example above furthermore shows that possible bifurcations depend not only on the combinatorial data, but also on the on the analytic data.
   \begin{figure}[h]%
 \centering
 \huge
     \resizebox{!}{4cm}{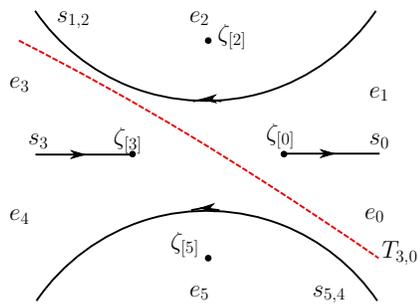}
 \caption{A class $\mathcal{C}$ with combinatorics $[0(1 \ 2)3](4 \ 5)$ having one sink $\zeta_{[3]}$, one source $\zeta_{[0]}$, and two centers $\zeta_{[2]}$ and $\zeta_{[5]}$. There is one $\alpha \omega$-zone and two homoclinic separatrices, so  $\dim (\CC)=2s+h=2(1)+2=4$. }
 \label{Nocelldecomp}
 \end{figure}
%---------------------------------------------------------
\clearpage
\section*{Appendix A. (with Tan Lei) - Landing Separatrices are Stable}
The main result in this appendix shows that the combinatorial structure given by landing separatrices is stable in some sense. Specifically, an equilibrium point which receives a landing separatrix cannot lose this separatrix under small perturbation, unless it is a multiple equilibrium point which splits.
\begin{definition}
The \emph{non-splitting set} $B_{\zeta^0}(\tilde{P}_0)$ for $\xi_{\tilde{P}_0}$ with respect to the \eqpt \ $\zeta^0$ is the subset of the sufficiently small neighborhood of $\xi_{\tilde{P}_0} \in \Xi_d$ such that  for the \eqpt \ $\zeta^0$ of $\tilde{P}_0$, there is exactly one \eqpt \ $\zeta$ for $\xi_{\tilde{P}}$ with $|\zeta^0-\zeta|< \delta$ (that is, $\text{mult}(\zeta^0)=\text{mult}(\zeta)$). The \emph{non-splitting set} $B(\tilde{P}_0)$ for $\xi_{\tilde{P}_0}$ is the intersection of $B_{\zeta^0}(\tilde{P}_0)$ for all $\zeta^0$.
\end{definition}
%The goal is to prove that under small perturbation, a landing separatrix must land at the same \eqpt \ as it did before perturbation, whenever the \eqpt \ in question does not split.  \par
The main theorem we aim to prove is the following:
\begin{theorem}
\label{landingstablethm}
Given $\xi_{\tilde{P}_0} \in \Xi_d$, if $s_{\ell}^0$ for $\xi_{\tilde{P}_0}$ lands at $\zeta^0$, $\tilde{P}_0(\zeta^0)=0$, then for every $\xi_{\tilde{P}}$ in the non-splitting set $B_{\zeta^0}(\tilde{P}_0)$ such that $\tilde{P}$ is "close enough" (to be defined) to $\tilde{P}_0$, then $s_{\ell}$ for $\xi_{\tilde{P}}$ lands at $\zeta$, $P(\zeta)=0$, where $\lim \limits_{\tilde{P} \rightarrow \tilde{P}_0}\zeta =\zeta^0$.
\end{theorem}
\noindent We will also need the following definition for inverses of rectifying coordinates:
\begin{definition}
For a (polynomial, $\infty$-germ) pair $(P, \gamma)$, define $\Psi_{P, \gamma}$ to be the inverse branch of $\Phi_P$ in a sector neighborhood of 0 as follows:
\begin{itemize}
\item
for $\gamma^+$ an outgoing $\infty$-germ, $\Psi_{P, \gamma^+}$ is defined on $D(\epsilon) \setminus \R^-$ and
coincides with $\gamma^+$ on $]0,\epsilon[$;
\item
for $\gamma^-$ an incoming $\infty$-germ, $\Psi_{P, \gamma^-}$ is defined on $D(\epsilon) \setminus \R^+$ and
coincides with $\gamma^-$ on $]-\epsilon, 0[$.
\end{itemize}
\end{definition}
\subsection{Preparation of Forms}
It is enough to consider $P_0$ and $P$ of the form
\begin{align}
P_0(z) &= z^k Q_0(z), \quad Q_0(0) \neq 0\nonumber \\
P(z) &= z^k Q(z), \quad Q(0) \neq 0.
\end{align}
First of all, any arbitrary $\xi_{\tilde{P}_0}$ with an \eqpt \ $\zeta^0$ of multiplicity $k$ is conformally conjugate to a unique $\xi_{P_0}$ with $P_0(z) = z^k Q_0(z)$, $Q_0(0) \neq 0$, by the translation $T_{\zeta^0}: z \mapsto z-\zeta^0$, and any $\xi_{\tilde{P}}$ in the non-splitting set $B_{\zeta^0}(\tilde{P}_0)$ has an \eqpt \ $\zeta$ of multiplicity $k$ such that $|\zeta^0-\zeta| < \delta$, so each $\tilde{P}$ in the non-splitting set can be uniquely conformally conjugated to $\xi_P$ with $P(z) = z^k Q(z)$, $Q(0) \neq 0$, by the translation $T_{\zeta}: z \mapsto z-\zeta$. Conjugating by translations does not change the asymptotic directions and hence labeling of the separatrices as compared to the original vector fields $\xi_{\tilde{P}_0}$ and $\xi_{\tilde{P}}$.\par
%Note also that by Lemma \ref{alphaforaffineconj} below, under these conjugacies, the $\alpha$-stability is not affected.
%\begin{lemma}[affine conjugacies \cite{Sent}]
We will write $P(z)=\left( 1+s(z)\right) P_0(z)$ and when we say that $P$ is \emph{close enough} to $P_0$, we mean that we have a uniform bound on s: $\|s\|_{\infty,U}\leq \epsilon''$, where $U$ is a restriction of $\Psi_0(S(\alpha))$ such that we avoid a neighborhood of the roots of $P$ and $P_0$ (except for 0). It is possible to demand such a uniform bound if $P$ and $P_0$ are close in terms of coefficients or roots by the following. Since $s(z)=\frac{P(z)}{P_0(z)}-1$ and $U$ avoids the roots of $P_0$, there is a uniform bound on $s$ on any compact subset of $U$ bounded away from 0 and $\infty$. Notice that near $\infty$, both $P\sim z^d$ and $P_0\sim z^d$, so $s \approx 0$ near $z=\infty$. Near $z=0$, the dominating terms are the constant terms, so $s(z) \approx a_0/a^0_0-1 \approx 0$ since we demand $P_0$ and $P$ are close in terms of coefficients. \par
Theorem \ref{landingstablethm}  hinges on the idea of $\alpha$-stability, as described in \cite{BT2007}. The notion of alpha-stability as presented in \cite{BT2007} is included here for completeness, and it should be compared to the notion of \emph{tolerant angle} in \cite{Sent}.
For $\alpha \in ]0,\frac{\pi}{2}[$, let us define a sector neighborhood of $\R^{\pm}$ by
\begin{equation}
S^+(\alpha)=\{ w \in \C^{\ast} \mid |\arg (w)|<\alpha   \}
\end{equation}
and
\begin{equation}
S^-(\alpha)=\{ w \in \C^{\ast} \mid |\pi-\arg (w)|<\alpha   \}.
\end{equation}
\begin{definition}[$\alpha$-stability as in \cite{BT2007}]
Given a polynomial P and an $\infty$-germ $\gamma$,  we say that $P$ is \emph{$(\alpha, \gamma)$-stable}, for $\alpha \in ]0,\frac{\pi}{2}[$, if $\Psi_{P, \gamma}$ extends holomorphically to the entire sector $S^+(\alpha)$ (if $\gamma$ is an outgoing germ), or $S^-(\alpha)$ (if $\gamma$ is an incoming germ). We will denote by $\Psi_{P, \gamma}:S^{\pm}(\alpha)\rightarrow \C $ this extension.
\end{definition}
\begin{remark}
We will only prove the theorem for \emph{outgoing} landing separatrices since the proof is completely analogous for \emph{incoming} separatrices.  Therefore, we will only be looking at positive sectors $S^+(\alpha)$, and will use the simpler notation $S(\alpha)$ for such a sector.
\end{remark}
\subsection{Landing Separatrices are Stable}
The idea of the main theorem is to show that if $s^0_{\ell}$ lands for $\xi_{P_0}$, then there exists a protective sector from infinity to 0 on the Riemann sphere where all trajectories that enter that sector converge to 0 (the equilibrium point). Small enough perturbations of $P_0$ guarantee that the corresponding $s_{\ell}$ is also trapped in this sector, and hence must converge to 0 as well.\par
We first show the existence of the protective sector $S(\alpha)$ in rectifying coordinates.
Assume that the separatrix $s^0_{\ell}$ is landing at the multiplicity $k$ equilibrium point $\zeta=0$ for $\xi_{P_0}$. We will show that any sequence approaching infinity in $S(\alpha)$ must approach $\zeta^0=0$.
\begin{proposition}
 If a separatrix $s^0_{\ell}$ is landing, then there exists an $\alpha$ such that $\xi_{P_0}$ is $(\alpha,\gamma^0_{\ell})$-stable.\par
 \end{proposition}
 \begin{proof}
 There are only three situations for landing separatrices (see Figure \ref{alphastability}):
 \begin{enumerate}
 \item
 The separatrix $s^0_{\ell}$ is on the boundary of two sepal zones (half planes). In this case, it is obvious that there exists such an $\alpha$.
 \item
 The separatrix $s^0_{\ell}$ is on the boundary of one sepal zone (half plane) and one $\alpha \omega$-zone (strip). There might be several strips between this strip and the next half plane. Such an $\alpha$ exists if we take the argument of the minimum of the partial sums of the analytic invariants in these strips (easier understood by referring to the figure).
 \item
 The separatrix $s^0_{\ell}$  is on the boundary of two $\alpha \omega$-zones (strips). The basin of the sink or source at which $s^0_{\ell}$ lands is a union of $n$ strips with an identification (cylinder), which we can unfold in the plane as a repeating sequence of strips. Let $A_i$ be the partial sums of the associated $\int_{T}\frac{dz}{P(z)}\in \HH$, where $T$ is a transversal joining the rightmost odd and even ends in a strip (for a sink). Let $\alpha = \min \limits_{i=1,\dots,n}\arg (A_i)$. Let $A_j$ be the partial sum associated to $\alpha$. No singularities fall inside the sector $S(\alpha)$. Indeed, the "worst" singularities are those at $ c \cdot A_n + A_j$, $c\in \mathbb{N}$. Since $\arg (A_n)\geq \alpha$, then $\arg (c \dot A_n + A_j)\geq \alpha$. The same must be done for the reverse partial sums. Take $\alpha$ to be the smallest from the forward and reverse minimum angles.
 \end{enumerate}
   \begin{figure}[h]%
 \centering
 \huge
     \resizebox{!}{4cm}{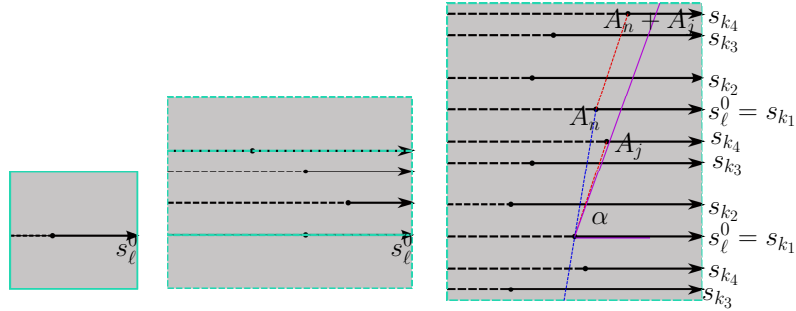}
 \caption{If $s^0_{\ell}$ is a landing separatrix, there exists an angle $\alpha$ such that $\xi_{P_0}$ is $(\alpha,\gamma^0_{\ell})$-stable. This can be seen in rectifying coordinates.  There are only three situations, depicted by the three figures above. The leftmost figure is the case where $s^0_{\ell}$ is on the boundary of two sepal zones (half planes). The middle figure is the case where $s^0_{\ell}$ is on the boundary of one sepal zone (half plane) and one $\alpha \omega$-zone (strip). The rightmost figure is the case where $s^0_{\ell}$  is on the boundary of two $\alpha \omega$-zones (strips). In all cases, it is easy to see that there is an $\alpha$ such that all singularities lie outside of the sector $S(\alpha)$. }
 \label{alphastability}
 \end{figure}
%---------------------
 \end{proof}
 \begin{proposition}
 The set $\Psi_0(S(\alpha))$ is completely contained in the basin of attraction for $\zeta^0=0$.
 \end{proposition}
 \begin{proof}
  The set $\Psi_0(S(\alpha))$ intersects the basin of $\zeta^0$ since it contains separatrix $s_{\ell}^0$. Furthermore, this set is connected, so if not entirely contained in the basin, it must intersect the boundary of the basin somewhere, which is not possible by Proposition \ref{BTClaimA} below.
 \end{proof}
 \begin{proposition}[\cite{BT2007}]
 \label{BTClaimA}
 $\Psi_0 \left(S(\alpha)\right)$ intersects neither the zeros nor the incoming $\infty$-germs of $P_0$.
  \end{proposition}
 \begin{proof}
The first part is due to that fact that it takes an infinite time to reach a zero.  For the second part, assume $\Psi_0(w_0)\in \gamma^-$ for some incoming $\infty$-germ $\gamma^-$ and some $w_0 \in S(\alpha')$ for $0<\alpha '<\alpha $ %we need this restriction so that defined on neighborhood of w0+t_0.
Then, by definition of incoming $\infty$-germs, the trajectory with initial point $\Psi_0(w_0)$ reaches $\infty$ at some positive finite time $t_0$. However, by uniqueness of solution, this trajectory coincides with $\Psi_0 (w_0+t_0)$. The fact that $\Psi_0$ is defined on a neighborhood of $w_0+t_0$ implies that $\Psi_0(w_0+t_0)\neq \infty$. This leads to a contradiction.
 \end{proof}
 We now need to show that any sequence that tends to infinity in $S(\alpha)$ tends to $\zeta^0$, not just those in the flow of $\xi_{P_0}$. \par
Let $S(\alpha)_R=  S(\alpha) \cap \{w \mid Re(w)>R \}$.
\begin{theorem}[Adapted from \cite{BT2007}]
There is a zero $\zeta^0$ of $P_0$ such that
\begin{equation}
\lim_{R\rightarrow +\infty} \Psi_0(S(\alpha)_R)=\zeta^0.
\end{equation}
\end{theorem}
\begin{proof}
For each $R>0$, $S(\alpha)_R$ is connected and hence it's closure in $\bar{\mathbb{C}}$ is connected. The intersection of nested continua $\bigcap \limits_{R>0}  \overline{S(\alpha)_R}$ is itself a continuum.  It must be a single point by Cantor's intersection theorem on complete metric spaces (after a coordinate change), and it must be $\zeta^0$ since $\Psi_0 \left(S(\alpha)\right)$ is contained in the basin for $\zeta^0$, and points in the flow of $\xi_0$ must tend to $\zeta^0$.
%In particular, $\lim_{t \rightarrow \infty, \ t \in \mathbb{R}_+}\Psi_0(t)=\zeta_0$, so it cannot be $\infty$ since no real time trajectory converges to infinity in infinite time. Therefore $\zeta^0$ is a zero of $P_0$.
\end{proof}
Summarizing the above in terms of what we need: If $s_{\ell}^0$ for $P_0$ lands at 0, a multiplicity $k$ equilibrium point, then there exists a protective sector $S(\alpha)$ such that all sequences going to infinity in  $S(\alpha)$ also limit at $\zeta^0 = 0$ (in the z-plane).\par
We will compare $P$ and $P_0$ in $S(\alpha)$ for $P_0$. Under the rectifying coordinates $\Phi_0$, $\dot{z}=P_0(z)$ conjugates to the constant vector field $\dot{w}=1$, and $\dot{z}=P(z)$ becomes $\dot{w}=1+s\circ \Psi_0(w)$.  \par
Since $s_{\ell}$ for $P$ is defined in a neighborhood of infinity, we know that there exists a solution $\gamma_{\ell}$ in a neighborhood of zero in $S(\alpha)$ for $\dot{w}=1+s\circ \Psi_0(w)$ which corresponds to part of the separatrix $s_{\ell}$. It enters the sector $S(\alpha)$ since perturbation does not change the asymptotic direction. We finish the proof of Theorem \ref{landingstablethm} by proving the following proposition.
\begin{proposition}
The trajectory $\gamma_{\ell}$ for $\xi_P$ mentioned above:
\begin{itemize}
\item[i.] $\gamma_{\ell}$ is defined for infinite forward time,
\item[ii.]  $\gamma_{\ell}$ does not leave $S(\alpha)$ for all time ($\gamma_{\ell}(t)\in S(\alpha)$ for all $t>0$), and
\item[iii.] $|\gamma_{\ell}(t)|\rightarrow \infty$ for $t\rightarrow \infty$
\end{itemize}
\end{proposition}
\begin{proof}
Item i. follows from the continuation of solutions theorem for ordinary differential equations (see  for instance ). Indeed $s(z):=\frac{P(z)}{P_0(z)}-1$ and $\Psi_0$ are holomorphic in $S(\alpha)$ (hence, so is $1+s\circ \Psi_0(w)$), and hence continuously differentiable in $\mathcal{R}:=S(\alpha)\times (-\infty, \ \infty)$. So the solution $\gamma_{\ell}(t)$ can be continued to a time interval $a\leq t < b$, where $b=+\infty$ unless one of the following two happen: (a) $|\gamma_{\ell}(t)|\rightarrow \infty$ as $t\rightarrow b^-<\infty$ (blows up in finite time), or (b) $\left(\gamma_{\ell}(t),t\right)$ leaves $\mathcal{R}$. Situation (a) cannot occur, since by $\dot{w}\approx 1$ uniformly, neither the real nor imaginary parts can blow up in finite time. Situation (b) cannot occur by item i.. Therefore, $\gamma_{\ell}(t)$ can be extended for infinite forward time, which proves item i.. Item ii. follows immediately from the fact that we can control $s$ uniformly so that $\dot{w}\approx 1$, since we can choose $P$ close enough to $P_0$ so that $\arg(1+s\circ \Psi_0(w))<\alpha$. Item iii. follows from item i. and again from the fact that we can control $s$ uniformly so that $\dot{w}\approx 1$.
\end{proof}
%%%%%%%%
\bibliographystyle{alphanum}
\bibliography{KealeyDias_ParameterSpace_LandingStable_withTanLei}{}
%%%%%%%%%%%%%
\noindent Department of Mathematics and Computer Science \\
Bronx Community College of the City University of New York\\
2155 University Avenue\\
Bronx, NY 10453\\
USA\\
\\
e-mail: kealey.dias@bcc.cuny.edu; kealey.dias@gmail.com\\
\\
%----------------------------
\end{document}